\documentclass{amsart}
\usepackage[latin1]{inputenc}
\usepackage{amssymb,amsmath,amsthm}
\usepackage{amsfonts}
\usepackage{ifthen}
\usepackage{paralist}
\usepackage{xypic}

\newcommand{\vect}[1]{\ensuremath\mathbf{#1}}

\newcommand{\subf}[1][]{\ifthenelse{\equal{#1}{}}{\ensuremath\leq}{\ensuremath\leq_{#1}}}
\newcommand{\subc}[1][]{\ifthenelse{\equal{#1}{}}{\ensuremath\preccurlyeq}{\ensuremath\preccurlyeq_{#1}}}
\newcommand{\fequiv}[1][]{\ifthenelse{\equal{#1}{}}{\ensuremath\equiv}{\ensuremath\equiv_{#1}}}

\newcommand{\eqclass}[2][]{\ifthenelse{\equal{#1}{}}{\ensuremath[{#2}]}{\ensuremath[{#2}]_{#1}}}

\theoremstyle{plain}
\newtheorem{theorem}{Theorem}[section]
\newtheorem{proposition}[theorem]{Proposition}
\newtheorem{lemma}[theorem]{Lemma}
\newtheorem{corollary}[theorem]{Corollary}
\newtheorem{clm}[theorem]{Claim}

\theoremstyle{definition}

\numberwithin{equation}{section}

\def\cl#1{{\mathcal #1}}
\def\al#1{{\mathbf #1}}

\def\Iso{{\text{\rm Iso}}}
\def\Im{{\text{\rm Im}}}
\def\GL{{\text{\rm GL}}}
\def\AGL{{\text{\rm AGL}}}
\def\FF{{\mathfrak F}}
\def\EE{{\mathfrak E}}
\def\id{{\rm id}}
\def\FFF{{\mathbb F}}

\def\subfunction{minor}

\begin{document}
\title{Equivalence of operations with respect to discriminator clones}
\thanks{This material is based upon work
supported by
the Hungarian National Foundation
for Scientific Research (OTKA)
grants no.\  T 048809 and K60148.
}
\author{Erkko Lehtonen}
\address[Erkko Lehtonen]{Institute of Mathematics \\
Tampere University of Technology \\
P.O. Box 553 \\ FI-33101 Tampere \\
Finland}
\email{erkko.lehtonen@tut.fi}
\author{\'Agnes Szendrei}
\address[\'Agnes Szendrei]{Department of Mathematics \\
University of Colorado \\
Boulder \\
CO 80309-0395 \\
USA\\ and
Bolyai Institute \\
Aradi v\'ertan\'uk tere 1 \\
H--6720 Szeged \\
Hungary }
\email{szendrei@euclid.colorado.edu}
\date{\today}
\begin{abstract}
For each clone $\cl{C}$ on a set $A$ 
there is an associated equivalence relation,
called $\cl{C}$-equivalence,
on the set of all operations on $A$, 
which relates two operations iff each one
is a substitution instance of the other 
using operations from $\cl{C}$.
In this paper we prove that 
if $\cl{C}$ is a discriminator clone on a
finite set, then there are only finitely many
$\cl{C}$-equivalence classes. Moreover,
we show that
the smallest discriminator clone is minimal with 
respect to this finiteness property.
For discriminator clones of Boolean functions
we explicitly describe the associated equivalence
relations.
\end{abstract}

\maketitle

\section{Introduction}

This paper is a study of how functions on a fixed
set can be classified using
their substitution
instances with inner functions taken from
a given set of functions.
In the theory of Boolean functions
several variants of this idea have been employed.
Harrison~\cite{Harrison} was interested in the number
of equivalence classes when $n$-ary Boolean functions
are identified if they are substitution
instances of each other with respect to 
the general linear group $\GL(n,\FFF_2)$
or the affine general linear group $\AGL(n,\FFF_2)$ ($\FFF_2$ is the two-element field).
Wang and Williams~\cite{WW} introduced classification
by Boolean minors to prove that the problem of 
determining  
the threshold order of a Boolean function 
is NP-complete.
They defined a Boolean function $g$ to be 
a {\it minor} of another Boolean function $f$ iff
$g$ can be obtained from $f$ by substituting for each variable of $f$ a variable, a negated variable,
or one of the constants $0$ or $1$. 
Wang~\cite{Wang} characterized various classes 
of Boolean functions by forbidden minors.
A more restrictive variant of Boolean minors, namely
when negated variables are not allowed,
was used in \cite{FH} and \cite{Zverovich} to characterize other classes of Boolean functions by forbidden minors.

In semigroup theory, Green's relation $R$, 
when applied to transformation semigroups $\cl{S}$, 
is another occurrence of the idea of classifying functions
by their substitution instances; namely, 
two transformations
$f,g\in\cl{S}$ are $R$-related iff
$f\bigl(h_1(x)\bigr)=g(x)$ and 
$g\bigl(h_2(x)\bigr)=f(x)$ for some
$h_1,h_2\in\cl{S}\cup\{\id\}$. 
Henno~\cite{Henno} generalized Green's relations to
Menger algebras (essentially, abstract clones),
and described Green's relations on the clone
$\cl{O}_A$ of all operations on $A$ for each set $A$.

The notions of $\cl{C}$-minor and $\cl{C}$-equivalence
where $\cl{C}$ is an arbitrary clone provide a
common framework for these results.
If $\cl{C}$ is a fixed clone on a set $A$, 
and $f,g$ are operations on $A$,
then $g$ is a $\cl{C}$-minor of $f$ 
if $g$ can be obtained
from $f$ by substituting operations from $\cl{C}$
for the variables of $f$,
and $g$ is $\cl{C}$-equivalent to $f$ if
$f$ and $g$ are both $\cl{C}$-minors of each other.
Thus, for example, the $R$-relation on $\cl{O}_A$
described in \cite{Henno} is nothing else than
$\cl{O}_A$-equivalence, and the concepts of Boolean minor
mentioned in the first paragraph 
are the special cases of the notion of
$\cl{C}$-minor where $\cl{C}$ is the essentially unary clone of Boolean functions generated by negation and the two constants, or by the two constants only.
For the least clone of Boolean functions, 
the essentially unary clone $\cl{P}$ of all projections, the $\cl{P}$-minor relation is investigated in 
\cite{CP}, and the classes of Boolean functions that  
are closed under taking $\cl{P}$-minors are characterized
in \cite{EFHH}.
The latter result is extended in \cite{Pippenger} to
classes of functions on finite sets that are closed
under taking $\cl{C}$-minors for arbitrary 
essentially unary clones $\cl{C}$.
The general notions of $\cl{C}$-minor and $\cl{C}$-equivalence, as introduced at the
beginning of this paragraph, first appeared in print in
\cite{ULM}, where the first author
studied the $\cl{C}$-minor quasiorder 
for clones $\cl{C}$ of monotone and linear operations.

The question this paper will focus on is the following.

\medskip

\noindent
{\bf Question.}
For which clones $\cl C$ on a finite set are there only
finitely many $\cl{C}$-equivalence classes of operations?

\medskip

The clones that have this property form a filter 
$\FF_A$ in the
lattice of clones on $A$
(cf.\ Proposition~\ref{basic_props2}).
Henno's 
result~\cite{Henno} (cf.\ Corollary~\ref{cor-O}) 
implies that $\cl{O}_A\in\FF_A$ if and only if
$A$ is finite.
Thus the filter $\FF_A$ is nonempty 
if and only if $A$ is finite.
The filter $\FF_A$ is proper if $|A|>1$, since 
the clone $\cl{P}_A$ of projections fails to belong to 
$\FF_A$.
The latter statement follows from the fact that
$\cl{P}_A$-equivalent operations have the same essential arity (i.e., depend on the same number of variables),
and on a set with more than one element there exist
operations of arbitrarily large essential arity.

In this paper we prove that every 
discriminator clone on a finite set $A$ 
belongs to $\FF_A$.
Furthermore, we show that if $|A|=2$,
then the members of $\FF_A$ are exactly the discriminator clones; thus in this case $\FF_A$ has a least member,
namely the smallest discriminator clone.
If $|A|>2$, then the analogous statements are no longer
true, because by a result of the first author 
in \cite{leht}, S{\l}upecki's clone belongs to 
$\FF_A$. S{\l}upecki's clone consists of all operations
that are either essentially unary or non-surjective,
therefore it is not a discriminator clone.
Thus for finite sets with three or more elements
the filter $\FF_A$ remains largely unknown.
However, we show that even in this case
the smallest discriminator clone is 
a minimal member of $\FF_A$.

In the last section of the paper 
we explicitly describe
the $\cl{C}$-equivalence and $\cl{C}$-minor
relations for 
discriminator clones of Boolean functions.

\section{Preliminaries}

Let $A$ be a fixed nonempty set. 
If $n$ is a positive integer, 
then by an $n$-ary operation on $A$ we
mean a function $A^n\to A$, 
and we will refer to $n$ as the
\emph{arity} of the operation. 
The set of all $n$-ary
operations on $A$ will be denoted by $\cl O_A^{(n)}$,
and we will write $\cl O_A$ for the set of all finitary operations on $A$.
For $1 \leq i \leq n$ the $i$-th $n$-ary \emph{projection} is the operation 
$p_i^{(n)}\colon A^n\to A,\ (a_1, \ldots, a_n) \mapsto a_i$. 

Every function $h\colon A^n\to A^m$ is uniquely determined
by the $m$-tuple of functions
$\vect{h}=(h_1,\ldots,h_m)$ where
$h_i=p_i^{(m)}\circ h\colon A^n\to A$ 
($i=1,\ldots,m$). 
In particular, $\vect{p}^{(n)}=(p_1^{(n)},\ldots,p_n^{(n)})$
corresponds to the identity function $A^n\to A^n$.
From now on we will identify each function
$h\colon A^n\to A^m$ with the corresponding $m$-tuple
$\vect{h}=(h_1,\ldots,h_m)\in (\cl O_A^{(n)})^m$ of $n$-ary
operations. 
Using this convention the
\emph{composition} of functions 
$\vect{h}=(h_1,\ldots,h_m)\colon A^n\to A^m$
and $\vect{g}=(g_1,\ldots,g_k)\colon A^m\to A^k$ 
can be written as
\[
\vect{g}\circ \vect{h}=
(g_1\circ \vect{h},\ldots,g_k\circ \vect{h})=
\bigl(g_1(h_1,\ldots,h_m),\ldots,g_k(h_1,\ldots,h_m)\bigr)
\]
where
\[
g_i(h_1, \ldots, h_m)(\vect{a}) = 
g_i\bigl(h_1(\vect{a}), \ldots, h_m(\vect{a})\bigr)
\qquad
\text{for all $\vect{a} \in A^n$ and for all $i$}.
\]

A \emph{clone} on $A$ is a subset $\cl C$ of $\cl O_A$
that contains the projections and is closed under composition; more precisely, this means that 
for all $m$, $n$ and $i$ ($1\le i\le n$),\ we have
$p_i^{(n)}\in\cl C$ and
whenever $g\in\cl C^{(m)}$ and $\vect{h}\in(\cl C^{(n)})^m$
then $g\circ\vect{h}\in\cl C^{(n)}$.
The clones on $A$ form a complete lattice under
inclusion. Therefore for each set $F\subseteq\cl O_A$
of operations there exists a smallest clone that contains
$F$, which will be denoted by $\langle F\rangle$
and will be referred to as the 
\emph{clone generated by} $F$.

Let $\cl C$ be a fixed clone on $A$.
For arbitrary operations $f\in\cl O_A^{(n)}$ and 
$g\in\cl O_A^{(m)}$ we say that 
\begin{itemize}
\item
$f$ is a
\emph{$\cl{C}$-\subfunction{}} of $g$, in symbols $f \subf[\cl{C}] g$, if $f = g\circ\vect{h}$ for some $\vect{h}\in (\cl C^{(n)})^m$; 
\item
$f$ and $g$ are \emph{$\cl{C}$-equivalent,} in symbols $f \fequiv[\cl{C}] g$, if $f \subf[\cl{C}] g$ and
$g \subf[\cl{C}] f$.
\end{itemize}
Some of the basic properties of the relations
$\subf[\cl{C}]$ and $\fequiv[\cl{C}]$
are summarized below.

\begin{proposition}
\label{basic_props}
Let $\cl C$ and $\cl C'$ be clones on $A$.
\begin{enumerate}
\item[{\rm(i)}]
$\subf[\cl{C}]$ is a quasiorder on $\cl{O}_A$.
\item[{\rm(ii)}]
$\fequiv[\cl{C}]$ is an equivalence relation 
on $\cl{O}_A$.
\item[{\rm(iii)}]
$\subf[\cl{C}]\,\subseteq\,\subf[\cl{C}']$ if and only if
$\cl{C}\subseteq\cl{C}'$.
\item[{\rm(iv)}]
$\fequiv[\cl{C}]\,\subseteq\,\fequiv[\cl{C}']$ if
$\cl{C}\subseteq\cl{C}'$.
\end{enumerate}
\end{proposition}

\begin{proof}
$f\subf[\cl{C}]f$ for all $f\in\cl{O}_A^{(n)}$ 
and $n\ge1$,
since $f=f\circ\vect{p}^{(n)}$ with
$\vect{p}^{(n)}\in(\cl{C}^{(n)})^n$,
as $\cl{C}$ contains the projections.
If $f\subf[\cl{C}] f' \subf[\cl{C}] f''$
where $f,f',f''$ have arities $k,m,n$, respectively,
then by definition, $f=f'\circ\vect{h}$ and 
$f'=f''\circ\vect{h}'$ for some 
$\vect{h}\in(\cl{C}^{(k)})^m$ and 
$\vect{h}'\in(\cl{C}^{(m)})^n$.
Thus $f=(f''\circ\vect{h}')\circ\vect{h}
=f''\circ(\vect{h}'\circ\vect{h})$ with
$\vect{h}'\circ\vect{h}\in(\cl{C}^{(k)})^n$
as $\cl{C}$ is closed under composition.
Hence $f\subf[\cl{C}] f''$.
This proves that $\subf[\cl{C}]$ is reflexive and transitive, establishing (i).
The claim in (ii) is an immediate consequence of (i).

It follows directly from the definitions that
for arbitrary clones $\cl C\subseteq \cl{C}'$ on $A$
we have $\subf[\cl{C}]\subseteq \subf[\cl{C}']$
and $\fequiv[\cl{C}]\subseteq \fequiv[\cl{C}']$.
This proves (iv) and the sufficiency in (iii).
To prove the necessity in (iii) notice that
$\{f\in\cl{O}_A:f\subf[\cl{C}]p_1^{(1)}\}=\cl{C}$.
This equality and the analogous equality for $\cl{C}'$
show that $\subf[\cl{C}]\,\subseteq\,\subf[\cl{C}']$ 
implies
$\cl{C}\subseteq\cl{C}'$.
\end{proof}

By definition, the equivalence relation $\fequiv[\cl{C}]$ is the intersection of
$\subf[\cl{C}]$ with its converse.
Therefore the quasiorder
$\subf[\cl{C}]$ induces a partial order on the set
$\cl{O}_A/{\fequiv[\cl{C}]}$ of $\cl{C}$-equivalence
classes. This partial order will be denoted by
$\preceq_\cl{C}$.
 
\begin{corollary}
\label{cor-nu}
If $\cl C$ and $\cl C'$ are clones on $A$ such that
$\cl{C}\subseteq\cl{C}'$, then 
\[
\nu_{\cl{C}',\cl{C}}\colon
\cl{O}_A/{\fequiv[\cl{C}]}\to\cl{O}_A/{\fequiv[\cl{C}']},
\qquad
f/{\fequiv[\cl{C}]}\mapsto f/{\fequiv[\cl{C}']}
\]
is an order preserving mapping of the poset
$(\cl{O}_A/{\fequiv[\cl{C}]};\preceq_\cl{C})$
onto 
$(\cl{O}_A/{\fequiv[\cl{C}']};\preceq_{\cl{C}'})$.
\end{corollary}

\begin{proof}
$\nu_{\cl{C}',\cl{C}}$ is well defined by 
Proposition~\ref{basic_props}~(iv), and order preserving by
Proposition~\ref{basic_props}~(iii).
The surjectivity of $\nu_{\cl{C}',\cl{C}}$ is clear from
its definition.
\end{proof}

By definition, $\nu_{\cl{C}',\cl{C}}$ 
($\cl{C}\subseteq\cl{C}'$)
maps each $\cl{C}$-equivalence class to the
$\cl{C}'$-equivalence class containing it.
Therefore
\begin{equation}
\label{comp-nu}
\nu_{\cl{C}'',\cl{C}} = 
\nu_{\cl{C}'',\cl{C}'}\circ\nu_{\cl{C}',\cl{C}}
\qquad
\text{if}
\qquad
\cl{C}\subseteq\cl{C}'\subseteq\cl{C}''.
\end{equation}

Now we will assume that $A$ is finite, and will
discuss some basic facts on
clones $\cl{C}$ for which 
$\fequiv[\cl{C}]$ has finite index in $\cl{O}_A$
(that is, the number of
$\cl C$-equivalence classes of operations on $A$
is finite). 
We will need the following notation.
If $\cl C$ is a clone on $A$ and $B$ is a nonempty
subset of $A$ such that every operation in $\cl C$
preserves $B$, then by restricting all operations
in $\cl C$ to $B$ we get a clone on $B$, which we will
denote by $\cl{C}|_B$.

\begin{proposition}
\label{basic_props2}
Let $\cl C$ be a clone on a finite set $A$.
\begin{enumerate}
\item[{\rm(i)}]
$\fequiv[\cl{C}]$ has finite index in $\cl{O}_A$
if and only if there exists an integer $d>0$ such that 
every operation on $A$ is $\cl C$-equivalent
to a $d$-ary operation on $A$.
\item[{\rm(ii)}]
If $\fequiv[\cl{C}]$ has finite index in $\cl{O}_A$,
then $\fequiv[\cl{C}']$ has finite index in $\cl{O}_A$
for every clone $\cl{C}'$ that contains $\cl C$.
\item[{\rm(iii)}]
If $\fequiv[\cl{C}]$ has finite index in $\cl{O}_A$
and $B$ is a nonempty subset of $A$ such that 
every operation in $\cl C$ preserves $B$, then
$\fequiv[\cl{C}|_B]$ has finite index in $\cl{O}_B$.
\end{enumerate}
\end{proposition}

\begin{proof}
(i) The number of $d$-ary operations on $A$ is finite,
since $A$ is finite. Therefore
if every operation on $A$ is $\cl C$-equivalent
to a $d$-ary operation on $A$, 
then $\fequiv[\cl{C}]$ has finite index in $\cl{O}_A$. 
Conversely, assume that $\fequiv[\cl{C}]$ has finite index in $\cl{O}_A$, and select a transversal $T$ for the blocks
of $\fequiv[\cl{C}]$. Since $T$ is finite, there is a $d>0$
such that every operation in $T$ is at most $d$-ary.
Now we will argue that for each operation $f\in T$, the $d$-ary operation $f^*$ obtained by adding 
fictitious variables to $f$ 
is $\cl{C}$-equivalent to $f$.
If $f$ is $k$-ary ($k\le d$), then
$f^*=f\circ(p_1^{(d)},\ldots,p_k^{(d)})$,
so $f^*\subf[\cl{C}] f$.
Since $(p_1^{(d)},\ldots,p_k^{(d)})\circ
(p_1^{(k)},\ldots,p_k^{(k)},
p_k^{(k)},\ldots,p_k^{(k)})=\vect{p}^{(k)}$,
we also get that 
$f^*\circ(p_1^{(k)},\ldots,p_k^{(k)},
p_k^{(k)},\ldots,p_k^{(k)})
=f\circ\vect{p}^{(k)}=f$,
so $f\subf[\cl{C}] f^*$.
Thus every operation on $A$ is $\cl{C}$-equivalent
to one of the $d$-ary operations $f^*$, $f\in T$.

(ii) follows immediately from 
Proposition~\ref{basic_props}~(iv).

(iii) Suppose that 
$\fequiv[\cl{C}]$ has finite index in $\cl{O}_A$.
By (i) there is an integer $d>0$ 
such that every operation
on $A$ is $\cl{C}$-equivalent to a $d$-ary
operation on $A$.
Now assuming
that $B$ is a nonempty subset of $A$ such that 
every operation in $\cl C$ preserves $B$
we will show that every operation
on $B$ is $\cl{C}|_B$-equivalent to a $d$-ary
operation on $B$.

Let $g$ be an $n$-ary operation on $B$.
Extend $g$ arbitrarily to an $n$-ary operation 
$f$ on $A$. Thus $f$ preserves $B$ and
$f|_B=g$.
By our assumption on $\cl{C}$, $f$ is $\cl{C}$-equivalent 
to a $d$-ary operation $f'$ on $A$.
Hence there exist $\vect{h}\in(\cl{C^{(d)}})^n$ and
$\vect{h}'\in(\cl{C}^{(n)})^d$ such that
$f'=f\circ\vect{h}$ and $f=f'\circ\vect{h}'$.
Since $f$ preserves $B$ (by construction) and 
the operations in $\cl C$ preserve $B$ (by assumption),
$f'=f\circ\vect{h}$ also preserves $B$.
Thus
$f'|_B=f|_B\circ\vect{h}|_B$ and 
$f|_B=f'|_B\circ\vect{h}'|_B$ where
all operations in 
$\vect{h}|_B$ and
$\vect{h}'|_B$ belong to $\cl{C}|_B$.
This proves that $g=f|_B$ is $\cl{C}|_B$-equivalent 
to the $d$-ary operation $f'|_B$.
\end{proof}

\section{The relation $\subf[\cl{C}]$ for discriminator
clones $\cl C$}

Let $A$ be an arbitrary set.
The \emph{discriminator function} on $A$ is the ternary
operation $t$ defined as follows:
\[
t(x,y,z) =
\begin{cases}
z, & \text{if $x = y$,} \\
x, & \text{otherwise}
\end{cases}
\qquad(x,y,z\in A).
\]
A clone on $A$ will be called a 
\emph{discriminator clone} if it contains $t$.

Let $\cl{C}$ be a clone on $A$.
An $n$-ary operation $f$ on $A$ is said to be
{\it locally in $\cl{C}$} if for every finite
subset $U$ of $A^n$ there exists an 
$n$-ary operation $g$ in $\cl{C}$ such that
$f(u)=g(u)$ for all $u\in U$.
The clone $\cl{C}$ is called {\it locally closed}
if $f\in\cl{C}$ for every operation $f$ that is locally
in $\cl{C}$.
It is easy to see from this definition that
if $A$ is finite, then every clone on $A$ is 
locally closed.
Examples of locally closed clones on an infinite set
$A$ include the clone of projections and the clone of all
operations on $A$.

Throughout this section $\cl{C}$ will be a 
locally closed discriminator clone on a set $A$,
and $\al A$ will denote the algebra $(A;\cl{C})$.
An isomorphism between subalgebras of $\al A$ is
called an \emph{internal isomorphism} of $\al A$.
We will use the notation $\Iso(\al A)$
for the family of all internal isomorphisms of $\al A$.

$\Iso(\al A)$ is a set of partial bijections that 
acts coordinatewise on $A^n$ for all $n\ge1$ as follows:
if $\vect{a}=(a_1,\ldots,a_n)\in A^n$, 
$\iota\in\Iso(\al A)$, and each $a_i$ is in the domain
of $\iota$, then 
$\iota(\vect{a})=\bigl(\iota(a_1),\ldots,\iota(a_n)\bigr)$;
otherwise $\iota(\vect{a})$ is undefined.
We will follow the convention that when we 
talk about elements $\iota(\vect{a})$ ``for some
[all] $\iota\in\Iso(\al A)$'' we will always
mean ``for some
[all] $\iota\in\Iso(\al A)$ for which $\iota(\vect{a})$
is defined''.

Since $\Iso(\al A)$ is closed under composition and 
inverses, the relation $\sim_\cl{C}$ on $A^n$ defined 
for all $\vect{a},\vect{b}\in A^n$ by
\[
\vect{a}\sim_\cl{C}\vect{b}
\quad\Leftrightarrow\quad
\vect{b}=\iota(\vect{a})
\text{ for some $\iota\in\Iso(\al A)$}
\]
is an equivalence relation whose blocks are the
$\Iso(\al A)$-orbits
\[
\vect{a}/{\sim_\cl{C}}=\{\iota(\vect{a}):\iota\in\Iso(\al A)\},
\quad \vect{a}\in A^n.
\]
We will choose and fix a transversal $T_n$ for the
blocks of $\sim_\cl{C}$ in $A^n$.

For an $n$-tuple $\vect{a}=(a_1,\ldots,a_n)$
let $\al{S}^\cl{C}_\vect{a}$ denote the subalgebra of
$\al A$ generated by
the set $\{a_1,\ldots,a_n\}$ of coordinates of $\vect{a}$.
Now let $\vect{a}\in A^n$ and  $\vect{b}\in A^m$ be
such that $\al{S}^\cl{C}_\vect{b}\le\al{S}^\cl{C}_\vect{a}$;
in other words, $\vect{b}\in(\al{S}^\cl{C}_\vect{a})^m$.
If $\iota_1,\iota_2\in\Iso(\al A)$ are internal
isomorphisms of $\al A$
such that $\iota_1(\vect{a})=\iota_2(\vect{a})$,
then $\iota_1,\iota_2$ agree on a generating set of
$\al{S}^\cl{C}_\vect{a}$. 
Thus $\iota_1,\iota_2$ are defined and
agree on $\al{S}^\cl{C}_\vect{a}$, and hence on
$\al{S}^\cl{C}_\vect{b}$.
This implies that 
$\iota_1(\vect{b})=\iota_2(\vect{b})$.
Thus
\[
\Phi^\cl{C}_{\vect{b},\vect{a}}\colon 
\vect{a}/{\sim_\cl{C}}\to\vect{b}/{\sim_\cl{C}},
\quad
\iota(\vect{a})\mapsto\iota(\vect{b})
\text{ for all $\iota\in\Iso(\al A)$}
\]
is a well-defined mapping of the $\sim_\cl{C}$-block
of $\vect{a}$ onto the $\sim_\cl{C}$-block of $\vect{b}$.
Notice that $\Phi^\cl{C}_{\vect{b},\vect{a}}$ is the unique
mapping $\vect{a/{\sim_\cl{C}}}\to\vect{b}/{\sim_\cl{C}}$ that 
sends $\vect{a}$ to $\vect{b}$ and preserves
all internal isomorphisms of $\al A$.
 
\begin{lemma}
\label{lm-h}
Let $\cl{C}$ be a locally closed
discriminator clone on a set $A$.
The following conditions on a function $\vect{h}\colon A^n\to A^m$ are equivalent:
\begin{enumerate}
\item[{\rm(a)}]
$\vect{h}\colon A^n\to A^m$ belongs to
$(\cl C^{(n)})^m$.
\item[{\rm(b)}]
$\vect{h}$ preserves the internal isomorphisms of $\al A$; 
that is, 
\[
\vect{h}\bigl(\iota(\vect{a})\bigr)=
\iota\bigl(\vect{h}(\vect{a})\bigr)
\quad
\text{for all $\iota\in\Iso(\al A)$}.
\]
\item[{\rm(c)}]
For each $n$-tuple $\vect{c}\in T_n$
there exists an $m$-tuple
$\vect{d}$ with $\al{S}^\cl{C}_\vect{d}\le \al{S}^\cl{C}_\vect{c}$
such that the restriction of $\vect{h}$ to
$\vect{c}/{\sim_\cl{C}}$ is the mapping $\Phi^\cl{C}_{\vect{d},\vect{c}}$.
\end{enumerate}
\end{lemma}

\begin{proof}
Since 
$\cl{C}$ is a locally closed clone, 
therefore $\cl{C}$ is the clone 
of local term operations 
of the algebra $\al A=(A;\cl{C})$.
The assumption that $t\in\cl{C}$,
combined with a theorem of 
Baker and Pixley~\cite{BP}, implies the
following well-known claim.

\begin{clm}
\label{bp} 
An operation $g\in\cl{O}_A$ belongs to $\cl{C}$
if and only if $g$ preserves all internal isomorphisms
of $\al A$.
\end{clm}

This implies that an analogous statement holds for
$m$-tuples of operations as well. 
Hence conditions (a) and (b) are equivalent.
It remains to show that conditions (b) and (c) are
equivalent.

First we will show that (b) $\Rightarrow$ (c).
Let $\vect{h}\in(\cl{C}^{(n)})^m$,
and let $\vect{c}\in T_n$.
Since $\vect{h}$
preserves all internal isomorphisms 
of $\al A$, it preserves, in particular, the identity
automorphism of each subalgebra of $\al A$. 
Hence $\vect{h}$ preserves all subalgebras of $\al A$.
This implies that the coordinates of the $m$-tuple 
$\vect{d}=\vect{h}(\vect{c})$ are in $\al{S}^\cl{C}_\vect{c}$.
Hence $\al{S}^\cl{C}_\vect{d}\le\al{S}^\cl{C}_\vect{c}$.
Moreover,
\[
\vect{h}\bigl(\iota(\vect{c})\bigr)=
\iota\bigl(\vect{h}(\vect{c})\bigr)=
\iota(\vect{d})=
\Phi^\cl{C}_{\vect{d},\vect{c}}\bigl(\iota(\vect{c})\bigr)
\quad
\text{for all $\iota\in\Iso(\al A)$}.
\]
This shows that $\vect{h}|_{\vect{c}/{\sim_\cl{C}}}$
coincides with $\Phi^\cl{C}_{\vect{d},\vect{c}}$, as claimed in
(c).

To prove the implication (c)~$\Rightarrow$~(b)
assume that $\vect{h}$ satisfies 
condition (c), and let $\kappa$ be an internal 
isomorphism of $\al A$. 
We have to show that $\vect{h}$
preserves $\kappa$. 
Let $\vect{a}$ be an arbitrary element of $A^n$
such that $\kappa(\vect{a})$ is defined, 
and let $\vect{c}$ be the representative of
the orbit $\vect{a}/{\sim_\cl{C}}$ in $T_n$.
There exists $\iota\in\Iso(\al A)$ such that 
$\vect{a}=\iota(\vect{c})$. 
Hence $\kappa(\vect{a})=(\kappa\circ\iota)(\vect{c})$.
Since $\vect{h}$ satisfies condition (c), there exists
$\vect{d}\in A^m$ with 
$\al{S}^\cl{C}_\vect{d}\le\al{S}^\cl{C}_\vect{c}$
such that the equality
$\vect{h}(\lambda(\vect{c}))=
\Phi^\cl{C}_{\vect{d},\vect{c}}(\lambda(\vect{c}))$
holds for all $\lambda\in\Iso(\al A)$.
Using this equality for $\lambda=\kappa\circ\iota$ and $\lambda=\iota$ (2nd and 6th equalities below),
the definition of $\Phi^\cl{C}_{\vect{d},\vect{c}}$ (3rd and 5th equalities), and the relationship between $\vect{a}$
and $\vect{c}$ (1st and 7th equalities),
we get that
\begin{multline}
\vect{h}\bigl(\kappa(\vect{a})\bigr)=
\vect{h}\bigl((\kappa\circ\iota)(\vect{c})\bigr)=
\Phi^\cl{C}_{\vect{d},\vect{c}}
    \bigl((\kappa\circ\iota)(\vect{c})\bigr)=
(\kappa\circ\iota)(\vect{d})\\
=
\kappa\bigl(\iota(\vect{d})\bigr)=
\kappa\bigl(\Phi^\cl{C}_{\vect{d},\vect{c}}
     (\iota(\vect{c}))\bigr)=
\kappa\bigl(\vect{h}(\iota(\vect{c}))\bigr)=
\kappa\bigl(\vect{h}(\vect{a})\bigr).\notag
\end{multline}
This proves that $\vect{h}$ preserves $\kappa$, and hence
completes the proof of the lemma.  
\end{proof}

\begin{theorem}
\label{le_C}
Let $\cl{C}$ be a locally closed
discriminator clone on a set $A$.
The following conditions on $f\in\cl{O}_A^{(n)}$ and
$g\in\cl{O}_A^{(m)}$ are equivalent:
\begin{enumerate}
\item[{\rm(a)}]
$f\subf[\cl{C}]g$.
\item[{\rm(b)}]
For each $\sim_\cl{C}$-block $P=\vect{c}/{\sim_\cl{C}}$\ \  
$(\vect{c}\in T_n)$ in $A^n$ there exists a 
$\sim_\cl{C}$-block $Q=\vect{d}/{\sim_\cl{C}}$ in $A^m$
such that $\al{S}^\cl{C}_\vect{d}\le\al{S}^\cl{C}_\vect{c}$ and 
$f|_P=g|_Q\circ\Phi^\cl{C}_{\vect{d},\vect{c}}$.
\end{enumerate}
\end{theorem}

\begin{proof}
(a)~$\Rightarrow$~(b).
If $f\subf[\cl{C}]g$, then $f=g\circ\vect{h}$ for some
$\vect{h}\in(\cl{C}^{(n)})^m$. 
Lemma~\ref{lm-h} shows that for each 
$\vect{c}\in T_n$ there exists $\vect{d}\in A^m$ with
$\al{S}^\cl{C}_\vect{d}\le\al{S}^\cl{C}_\vect{c}$ 
such that by restricting 
$\vect{h}$ to $P=\vect{c}/{\sim_\cl{C}}$
we get the function 
$\vect{h}|_P=\Phi^\cl{C}_{\vect{d},\vect{c}}\colon 
P \to Q=\vect{d}/{\sim_\cl{C}}$.
Thus $f|_P=(g\circ\vect{h})|_P=g|_Q\circ\vect{h}|_P
=g|_Q\circ\Phi^\cl{C}_{\vect{d},\vect{c}}$.

(b)~$\Rightarrow$~(a).
Assume that condition (b) holds for $f$ and $g$.
For each $\vect{c}\in T_n$ fix a tuple 
$\vect{d}=\vect{d}_\vect{c}$ whose existence is
postulated in condition (b).
Since every $\sim_\cl{C}$-block 
$P$ in $A^n$ is of the form
$P=\vect{c}/{\sim_\cl{C}}$ for a unique $\vect{c}\in T_n$,
there is a (well-defined) function 
$\vect{h}\colon A^n\to A^m$ such that 
$\vect{h}|_P=\Phi^\cl{C}_{\vect{d}_\vect{c},\vect{c}}$ for all
$\sim_\cl{C}$-blocks $P$ in $A^n$.
Lemma~\ref{lm-h} implies that 
$\vect{h}\in(\cl{C}^{(n)})^m$.
Moreover, we have $f=g\circ\vect{h}$, because
condition (b) and the construction of $\vect{h}$ yield that
$f|_P=g|_Q\circ\Phi^\cl{C}_{\vect{d}_\vect{c},\vect{c}}=
g|_Q\circ\vect{h}|_P=(g\circ\vect{h})|_P$
for all $\sim_\cl{C}$-blocks $P$ in $A^n$.
Thus $f\subf[\cl{C}]g$.
\end{proof}

We conclude this section by applying Theorem~\ref{le_C}
to the clone $\cl{C}=\cl{O}_A$, which is clearly
a locally closed discriminator clone for every set $A$.
If $\cl{C}=\cl{O}_A$, then the algebra 
$\al{A}=(A;\cl{O}_A)$ has no proper subalgebras and no
nonidentity automophisms.
Therefore $\vect{a}/{\sim_{\cl{O}_A}}=\{\vect{a}\}$ and   
$\al{S}_\vect{a}^{\cl{O}_A}=A$ for all
$\vect{a}\in A^n$, $n\ge1$. 
Moreover, 
$\Phi_{\vect{b},\vect{a}}^{\cl{O}_A}$ is the unique 
mapping $\{\vect{a}\}\to\{\vect{b}\}$.
Thus condition (b) in Theorem~\ref{le_C} for
$\cl{C}=\cl{O}_A$ requires the following:
for every block
$P=\{\vect{c}\}$ in $A^n$, if 
$f|_P\colon \{\vect{c}\}\to\{r\}$, then 
there exists a block $Q=\{\vect{d}\}$ in $A^m$
such that 
$g|_Q\colon \{\vect{d}\}\to\{r\}$;
that is, every element $r$ that is in the range 
$\Im(f)$ of $f$
is also in the range 
$\Im(g)$ of $g$.
Hence Theorem~\ref{le_C} yields the following result
from \cite{Henno} (see also \cite{leht}):

\begin{corollary}
\label{cor-O}
Let $A$ be a set.
For arbitrary operations $f\in\cl{O}_A^{(n)}$ and
$g\in\cl{O}_A^{(m)}$,
$$
f\subf[\cl{O}_A]g
\qquad
\text{if and only if}
\qquad
\Im(f)\subseteq\Im(g).
$$
\end{corollary}

Further applications of Theorem~\ref{le_C}
will appear in Sections~4 and 5.

\section{Finiteness and minimality}

Let $A$ be a finite set, and
let $\FF_A$ denote the family of all clones $\cl{C}$
on $A$ such that there are only finitely many
$\cl{C}$-equivalence classes of operations on $A$ 
(that is, $\fequiv[\cl{C}]$ has finite index 
in $\cl{O}_A$).
Lemma~\ref{basic_props2}~(ii) shows that $\FF_A$
is a filter in the lattice of all clones on $A$.
By Corollary~\ref{cor-O}, the clone
$\cl{O}_A$ belongs to $\FF_A$.

Our goal in this section is to prove that
all discriminator clones belong to $\FF_A$.
Since $\FF_A$ is a filter,
it will be sufficient to show that the
smallest discriminator clone 
$\cl{D}=\langle t\rangle$ belongs to $\FF_A$.
We will also prove that $\cl{D}$ is a minimal
member of $\FF_A$, that is, no proper subclone
of $\cl{D}$ belongs to $\FF_A$.

Our main result is

\begin{theorem}
\label{main-thm}
Let $A$ be a finite set of cardinality $|A|=k\ge2$,
and let $\cl{D}$ be the clone
generated by the discriminator function on $A$.
For $d= k^k-k^{k-1}+1$,
every operation on $A$ is $\cl{D}$-equivalent
to a $d$-ary operation on $A$.
\end{theorem}

This theorem, combined with
Lemma~\ref{basic_props2}~(i) and (ii), 
immediately implies the corollary below
which states 
that all discriminator clones
belong to $\FF_A$.

\begin{corollary}
\label{main-cor}
For each discriminator clone $\cl{C}$ on a finite set
$A$ the equivalence relation $\fequiv[\cl{C}]$
has finite index in $\cl{O}_A$.
\end{corollary}

For the proof of Theorem~\ref{main-thm} we will
use the description of $\subf[\cl{D}]$ in
Section~3.
Recall that since $A$ is finite, all clones on
$A$ are locally closed.
We will denote the symmetric group on $n$ letters
by $S_n$.

\begin{proof}[Proof of Theorem~\ref{main-thm}]
Let $\al A=(A;\cl{D})$.
We may assume without loss of generality that
$A=\{1,2,\ldots,k\}$.
The discriminator function preserves all bijections
between any two subsets of $A$ of the same size.
Therefore 
\begin{enumerate}
\item[(1)]
all subsets of $A$ are subalgebras of $\al A$, and
\item[(2)]
$\Iso(\al A)$ is the set of all bijections 
$B\to C$ such that $B,C\subseteq A$ and $|B|=|C|$.
\end{enumerate}
Hence for each $\vect{a}=(a_1,\ldots,a_n)\in A^n$
\begin{enumerate}
\item[(3)]
$\al{S}^\cl{D}_\vect{a}$ is the set of coordinates of
$\vect{a}$, and
\item[(4)]
the  $\Iso(\al A)$-orbit ($\sim_\cl{D}$-block)
of $\vect{a}$ is
\[
\vect{a}/{\sim_\cl{D}}=\{(b_1,\ldots,b_n)\in A^n:
\text{$b_i=b_j$ $\Leftrightarrow$ $a_i=a_j$
holds for all $i,j$}\}.
\]
\end{enumerate}
The number of distinct coordinates of
$\vect{a}$ will be called the \emph{breadth 
of $\vect{a}$}.
It follows from (4) that all tuples in a 
$\sim_\cl{D}$-block
$P=\vect{a}/{\sim_\cl{D}}$ have the same breadth; 
this number will be called the \emph{breadth of $P$}, 
and will be denoted by $\nu(P)$.
Another consequence of (4) is that
\begin{enumerate}
\item[(5)]
every $\sim_\cl{D}$-block $P$ of breadth $r$ in $A^n$
can be represented by a tuple $\vect{c}=(c_1,\ldots,c_n)$
such that $\{c_1,\ldots,c_n\}=\{1,\ldots,r\}$;
\item[(6)]
moreover, this representative $\vect{c}$ is unique
if we require in addition that the first occurrences of
$1,\ldots,r$ among $c_1,\ldots,c_n$ appear in increasing order; that is, 
if the first
occurrence of $i$ ($1\le i\le r$) in $(c_1,\ldots,c_n)$
is $c_{j_i}$ for each 
$i$, then $j_1<j_2<\cdots<j_r$.
\end{enumerate}
Thus the $n$-tuples $\vect{c}$ that satisfy the conditions
described in (5)--(6) form a transversal for
the $\sim_\cl{D}$-blocks of $A^n$. We will select this transversal to be $T_n$.

Let $\vect{c}\in T_n$. With the notation used in (5)--(6)
we get from (4) that the projection mapping
\[
\pi_P\colon P\to P_r,
\quad
(a_1,\ldots,a_n)\mapsto(a_{j_1},\ldots,a_{j_r})
\]
whose range $P_r$ is the unique 
$\sim_\cl{D}$-block of breadth $r$ in $A^r$ 
is bijective, and maps $\vect{c}$ to 
the $r$-tuple $\vec{r}=(1,\ldots,r)\in T_r$.

For a permutation $\sigma\in S_r$ the bijection
$P_r\to P_r, (x_1,\ldots,x_r)\mapsto
(x_{\sigma(1)},\ldots,x_{\sigma(r)})$ 
that permutes the coordinates of $P_r$ by $\sigma$
will be denoted by $\sigma^*$.

\begin{clm}
\label{proj-fg}
Let $f\in\cl{O}_A^{(n)}$ and $g\in\cl{O}_A^{(m)}$.
If for every $r$ $(1\le r\le k)$ and
for every $\sim_\cl{D}$-block $P$ of breadth $r$ in $A^n$
there exists a $\sim_\cl{D}$-block $Q$ of breadth $r$ in
$A^m$ such that 
\begin{equation}
\label{eq-proj-fg}
f|_P\circ\pi_P^{-1}=
(g|_Q\circ\pi_Q^{-1})\circ\sigma^*
\quad
\text{for some $\sigma\in S_r$,}
\end{equation}
then $f\subf[\cl{D}]g$.
\end{clm}

Suppose that the hypotheses of the claim
are satisfied. To prove that $f\subf[\cl{D}]g$
if suffices to verify that condition (b) in
Theorem~\ref{le_C} with $\cl{C}=\cl{D}$ holds.
Let
$P=\vect{c}/{\sim_\cl{D}}$ ($\vect{c}\in T_n$) be an 
arbitrary $\sim_\cl{D}$-block 
of breadth $r$ in $A^n$, and let
$Q=\vect{c}'/{\sim_\cl{D}}$ ($\vect{c}'\in T_m$) be a
$\sim_\cl{D}$-block in $A^m$ for which 
(\ref{eq-proj-fg}) holds.
Furthermore, let $\vec{r}=(1,\ldots,r)$,
and let $\vect{d}=\sigma(\vect{c}')$
be the image of $\vect{c}'$ under the
internal isomorphism $\sigma$ of $\al A$.

Notice that each one of the mappings $\pi_P$, $\pi_Q$,
and $\sigma^*$ are bijections between $\sim_\cl{D}$-blocks, and preserve the internal isomorphisms of $\al A$. Therefore the mapping 
$\pi_Q^{-1}\circ\sigma^*\circ\pi_P\colon P\to Q$
also preserves the internal isomorphisms of $\al A$.
The image of $\vect{c}$ under this mapping is
$\vect{d}$, as the following calculation shows:
\[
\pi_Q^{-1}\bigl(\sigma^*(\pi_P(\vect{c}))\bigr)=
\pi_Q^{-1}\bigl(\sigma^*(\vec{r})\bigr)=
\pi_Q^{-1}\bigl((\sigma(1),\ldots,\sigma(r))\bigr)=
\pi_Q^{-1}\bigl(\sigma(\vec{r})\bigr)=
\sigma(\vect{c}')=\vect{d},
\]
where the second to last equality holds, because
$\pi_Q^{-1}(\vec{r})=\vect{c}'$ and
$\sigma$ is an internal isomorphism of $\al A$.
Since $\Phi^\cl{D}_{\vect{d},\vect{c}}$
is the unique mapping $P\to Q$ that preserves the internal
isomorphisms of $\al A$ and sends $\vect{c}$ to
$\vect{d}$, we get that 
$\pi_Q^{-1}\circ\sigma^*\circ\pi_P=
\Phi^\cl{D}_{\vect{d},\vect{c}}$.
Thus the equality in (\ref{eq-proj-fg}) is equivalent to
$f|_P=
g|_Q\circ\pi_Q^{-1}\circ\sigma^*\circ\pi_P
=g|_Q\circ \Phi^\cl{D}_{\vect{d},\vect{c}}$.
The $m$-tuple $\vect{d}= \sigma(\vect{c}')$
clearly satisfies $Q=\vect{d}/{\sim_\cl{D}}$ and 
$\al{S}^\cl{D}_{\vect{d}}=\{1,\ldots,r\}=
\al{S}^\cl{D}_{\vect{c}}$ (see statement (3) above).
This shows that condition (b) in
Theorem~\ref{le_C} with $\cl{C}=\cl{D}$ holds,
and hence completes the proof of
Claim~\ref{proj-fg}.

\medskip

In Claim~\ref{proj-fg} $f|_P\circ\pi_P^{-1}$ and 
$g|_Q\circ\pi_Q^{-1}$ are both functions $P_r\to A$,
and condition (\ref{eq-proj-fg}) says that,
up to a permutation of the coordinates of $P_r$,
they are the same function.
For arbitrary functions $\phi,\psi\colon P_r\to A$
let
\[
\phi\approx\psi
\quad\Leftrightarrow\quad
\phi=\psi\circ\sigma^*
\text{ for some }\sigma\in S_r.
\]
In other words, $\phi\approx\psi$ iff
$\phi$ and $\psi$ are in the same orbit
under the action of the symmetric group $S_r$ 
on the set $A^{P_r}$ of all functions $P_r\to A$ by permuting the coordinates of $P_r$.
Hence $\approx$ is an equivalence relation on
$A^{P_r}$.
With this notation
condition (\ref{eq-proj-fg}) above 
can be restated to say
that  
$f|_P\circ\pi_P^{-1}$ and
$g|_Q\circ\pi_Q^{-1}$ are 
in the same $\approx$-block of $A^{P_r}$.

For arbitrary $n$-ary operation $f$ on $A$ ($n\ge1$)
and integer $r$ ($1\le r\le k$) let
$\EE_r(f)$ denote the set of $\approx$-blocks of
all functions $f|_P\circ\pi_P^{-1}$
as $P$ runs over all $\sim_\cl{D}$-blocks of breadth $r$ in $A^n$.

\begin{clm}
\label{main-clm2}
Let $f\in\cl{O}_A^{(n)}$ and $g\in\cl{O}_A^{(m)}$.
\begin{enumerate}
\item[{\rm(i)}]
If $\EE_r(f)\subseteq\EE_r(g)$ for all $r$ $(1\le r\le k)$,
then $f\subf[\cl{D}]g$.
\item[{\rm(ii)}]
If $\EE_r(f)=\EE_r(g)$ for all $r$ $(1\le r\le k)$,
then $f\fequiv[\cl{D}]g$.
\end{enumerate}
\end{clm}

Part (i) is a restatement of Claim~\ref{main-clm}
using the notation introduced after Claim~\ref{main-clm}.
Part (ii) is an immediate consequence of (i).

\medskip

Now let $N(k,r)$ denote the index of $\approx$ 
(i.e., the number of $\approx$-blocks) in $A^{P_r}$,
where $k=|A|$.
We will also use the Stirling numbers 
$S(d,r)$ of the second kind.
Since the $\sim_\cl{D}$-blocks of breadth $r$ in $A^d$
are in one-to-one correspondence with the
partitions of $\{1,\ldots,d\}$ into $r$ blocks,
$S(d,r)$ is the number of
$\sim_\cl{D}$-blocks of breadth $r$ in $A^d$.

\begin{clm}
\label{main-clm}
If $d$ is a positive integer such that
\begin{equation}
\label{eq-d}
N(k,r)\le S(d,r)
\quad
\text{for all $r$ with $2\le r\le k$,}
\end{equation}
then every operation $f\in\cl{O}_A$ is $\cl{D}$-equivalent
to a $d$-ary operation.
\end{clm}

Assume that (\ref{eq-d})
holds for $d$, and let $f$ be
an arbitrary operation on $A$,
say $f$ is $n$-ary. In view of Claim~\ref{main-clm2}~(ii)
it suffices to show that there exists a $d$-ary operation
$g$ on $A$ such that $\EE_r(g)=\EE_r(f)$ for all
$r$ ($1\le r\le k$).
Since the $\sim_\cl{D}$-blocks partition $A^d$, we may
define $g$ on each $\sim_\cl{D}$-block separately.

For the unique $\sim_\cl{D}$-block 
$Q=(1,\ldots,1)/{\sim_\cl{D}}$ of breadth $r=1$ in $A^d$
we define $g|_Q$ to be $f|_P\circ\pi_P^{-1}\circ\pi_Q$
where $P=(1,\ldots,1)/{\sim_\cl{D}}$ is the unique
$\sim_\cl{D}$-block of breadth $1$ in $A^n$.
This will ensure that $\EE_r(f)=\EE_r(g)$ holds for $r=1$.

If $2\le r\le k$, then $|\EE_r(f)|\le N(k,r)\le S(d,r)$, 
where the first inequality follows from the definition of
$\EE_r(f)$, while the second equality is our assumption.
Let $\phi_1,\ldots,\phi_s$ be a transversal
for the $\approx$-blocks in $\EE_r(f)$.
The inequality $s=|\EE_r(f)|\le S(d,r)$ ensures that
we can select $s$ distinct 
$\sim_\cl{D}$-blocks $Q_j$ ($j=1,\ldots,s$) of breadth
$r$ in $A^d$. 
Now for each $\sim_\cl{D}$-block 
$Q$ of breadth $r$ in $A^d$
we define $g|_Q=\phi_j\circ\pi_Q$ if $Q=Q_j$
($j=1,\ldots,s$), and  $g|_Q=\phi_1\circ\pi_Q$ 
otherwise.
Clearly, this will imply that $\EE_r(f)=\EE_r(g)$ holds for $r\ge2$. 
This completes the proof of the claim.

\medskip

To finish the proof of Theorem~\ref{main-thm} it remains to
show that (\ref{eq-d}) holds for $d=k^k-k^{k-1}+1$.

\begin{clm}
\label{imeq-d}
Condition (\ref{eq-d}) holds for $d=k^k-k^{k-1}+1$.
\end{clm}

If $k=2$, then $d=3$, and
the only value of $r$ to be considered is
$r=2$. 
It is straightforward to check that in this case
$N(k,r)=N(2,2)=3$ and 
$S(d,r)=S(3,2)=3$. Therefore
(\ref{eq-d}) holds for $k=2$.

From now on we will assume that $k \ge 3$. 
Let $2\le r\le k$.
We have $d > r$, because 
$d = k^{k-1}(k-1)+1 > (k-1)+1=k$.
The number of
equivalence relations on $\{1, 2, \ldots, d\}$ with 
exactly $r$ blocks is at least $r^{d - r}$, 
since the identity function 
$\{1, 2, \ldots, r\} \to \{1, 2, \ldots, r\}$ 
can be extended in $r^{d - r}$ different ways 
to a function 
$\{1, 2, \ldots, d\} \to \{1, 2, \ldots, r\}$ and these extensions have distinct kernels, which are equivalence relations on $\{1, 2, \ldots, d\}$ with exactly $r$ blocks. Thus, 
\[
r^{d - r}\le S(d,r).
\]

The number of functions $P_r \to A$ is
$k^{k(k-1)\cdots(k-r+1)}$, 
therefore 
\[
N(k,r) \leq k^{k(k-1)\cdots(k-r+1)}\le k^{k!}.
\] 
Since $k \geq 3$, we have $k! < k^{k-1}$ and 
$k \leq 2^{k-1}$. 
Thus we get the first inequality in
\[
k^{k!} \leq (2^{k-1})^{k^{k-1} - 1} = 
2^{k^k - k^{k-1} + 1 - k} =
2^{d-k}\leq 
r^{d - r}.
\]
The last inequality, 
$2^{d - k} \leq r^{d - r}$, 
follows from $2 \leq r \leq k$. 
Combining the displayed inequalities 
we get that $N(k,r)\le k^{k!}\le r^{d-r}\le S(d,r)$. 

This completes the proof of Theorem~\ref{main-thm}.
\end{proof}

Theorem~\ref{main-thm} shows that $\cl{D}$ belongs to
the filter $\FF_A$ of all clones $\cl{C}$ on $A$ for
which there are only finitely many $\cl{C}$-equivalence
classes of operations on $A$.
The next theorem will prove that if $|A|=2$, then
all members of $\FF_A$ are discriminator clones.
Hence in this case $\FF_A$ is a principal filter
in the lattice of clones on $A$ with least element
$\cl{D}$.

\begin{theorem}
\label{boolean-case}
For a two-element set $A$, if $\cl{C}$ is not 
a discriminator clone on $A$, then
$\fequiv[\cl{C}]$ has infinite index in $\cl{O}_A$.
\end{theorem}

\begin{proof}
We may assume without loss of generality that $A=\{0,1\}$.
The lattice of all clones on $\{0,1\}$ was described by
Post~\cite{Post}.
By inspecting Post's lattice one can see that if $\cl{C}$
is not a discriminator clone, then $\cl{C}$
is a subclone of one of the following clones:
\begin{itemize}
\item
the clone $\cl{L}$ of linear operations,
\item
the clone $\cl{M}$ of all operations that are 
monotone with respect to the partial order $0\le 1$,
\item
the clone $\cl{R}_0$ of all operations that preserve the binary relation $\rho_0=\{(0,0),\ (0,1),\ (1,0)\}$, and
\item
the clone $\cl{R}_1$ of all operations that preserve the binary relation $\rho_1=\{(1,1),\ (1,0),\ (0,1)\}$.
\end{itemize}
In view of Lemma~\ref{basic_props2}~(ii), to show that
$\fequiv[\cl{C}]$ has infinite index in $\cl{O}_A$
it suffices to verify that each one of the four 
equivalence relations 
$\fequiv[\cl{L}]$, $\fequiv[\cl{M}]$, and $\fequiv[\cl{R}_i]$ ($i=0,1$)
has infinite index in $\cl{O}_A$.
This will be done in the 
Claims~\ref{restr-lin}--\ref{restr-cross} below.

\medskip

\begin{clm}
\label{restr-lin}
The equivalence relation $\fequiv[\cl{L}]$ has infinite index in $\cl{O}_A$.
\end{clm}

It follows from a result in \cite[Proposition 5.9]{ULM} 
that if $\cl{L}$ is the clone of all linear operations
on $A=\{0,1\}$, then there exists an infinite sequence
of operations $g_n\in\cl{O}_A$ ($n=1,2,\ldots$)
such that $g_m\not\subf[\cl{L}] g_n$ whenever $m\not=n$.
This implies that the equivalence relation 
$\fequiv[\cl{L}]$ has infinite index in $\cl{O}_A$.

\begin{clm}
\label{restr-mon}
The equivalence relation $\fequiv[\cl{M}]$ has infinite index in $\cl{O}_A$.
\end{clm}

For $n \geq 1$ let $f_n$ be the $n$-ary linear operation 
$f_n(x_1,x_2\ldots,x_n) = x_1 + x_2 + \dots + x_n$
on $A=\{0,1\}$. 
Our claim will follow if we show that the operations
$f_n$ are in pairwise distinct  $\fequiv[\cl{M}]$-blocks.
To this end it will be sufficient to verify that
$f_m\subf[\cl{M}] f_n$ if and only if $m\le n$.
If $m\le n$, then 
$f_m = f_n(p_1^{(m)}, p_2^{(m)},\ldots,p_m^{(m)}, 0, \ldots, 0)$, where the projections $p_i^{(m)}$ and the constant function $0$ are members of $\cl{M}$.
Hence $f_m \subf[\cl{M}] f_n$. 

Conversely, assume that $f_m \subf[\cl{M}] f_n$.
By definition, this means that
there exists $\vect{h} \in (\cl{M}^{(m)})^n$ 
such that $f_m = f_n \circ \vect{h}$. 
Consider the chain $\vect{e}_0<\vect{e}_1<\cdots
<\vect{e}_m$ in $(A;\le)^m$ where
$\vect{e}_i=(1, \ldots, 1, 0, \ldots, 0)$
($0 \leq i \leq m$)
is the $m$-tuple whose
first $i$ coordinates are $1$ and 
last $m - i$ coordinates are $0$.
Since $\vect{h} \in (\cl{M}^{(m)})^n$, therefore
$\vect{h}$ is an order preserving mapping of
$(A;\le)^m$ to $(A;\le)^n$.
Thus
$\vect{h}(\vect{e}_0)\le\vect{h}(\vect{e}_1)\le
\cdots\le\vect{h}(\vect{e}_m)$
holds in $(A;\le)^n$.
Moreover, these elements are pairwise distinct, because
the calculation below shows that 
$f_n\bigl(\vect{h}(\vect{e}_i)\bigr)\not=
f_n\bigl(\vect{h}(\vect{e}_{i+1})\bigr)$
for each $i$ ($0 \leq i < m$); indeed,
\[
f_n\bigl(\vect{h}(\vect{e}_i)\bigr) = 
(f_n \circ \vect{h})(\vect{e}_i) = 
f_m(\vect{e}_i) \neq f_m(\vect{e}_{i+1}) = 
(f_n \circ \vect{h})(\vect{e}_{i+1}) = f_n\bigl(\vect{h}(\vect{e}_{i+1})\bigr).
\]
This proves that 
$\vect{h}(\vect{e}_0)<\vect{h}(\vect{e}_1)<
\cdots<\vect{h}(\vect{e}_m)$ 
is a chain of length $m$ 
in $(A;\le)^n$.
In the partially ordered set $(A;\le)^n$
the longest chain has length $n$, therefore 
$m\le n$.

\begin{clm}
\label{restr-cross}
The equivalence relation $\fequiv[\cl{R}_\ell]$
$(\ell=0,1)$ has infinite index in $\cl{O}_A$.
\end{clm}

The clone $\cl{R}_1$ can be obtained from $\cl{R}_0$
by switching the role of the two elements of $A=\{0,1\}$,
therefore it suffices to prove the claim for $\ell=0$. 
As in the preceding claim, we let
$f_n$ ($n\ge1$) be the $n$-ary linear operation 
$f_n(x_1,x_2\ldots,x_n) = x_1 + x_2 + \dots + x_n$
on $A$, and want to prove that
$f_m\subf[\cl{R}_0] f_n$ if and only if $m\le n$.
If $m\le n$, then $f_m \subf[\cl{R_0}] f_n$ follows
the same way as before, 
since the projections and the constant 
function $0$ are members of $\cl{R}_0$.

Now assume that $f_m \subf[\cl{R}_0] f_n$.
By definition, 
there exists $\vect{h}=(h_1,\ldots,h_n) 
\in (\cl{R}_0^{(m)})^n$ 
such that $f_m = f_n \circ \vect{h}$. 
Consider the $m$-tuples
$\vect{e}_i=(0, \ldots, 0,1, 0, \ldots, 0)\in A^m$ 
where the single $1$ occurs
in the $i$-th coordinate ($1 \leq i \leq m$).
Notice that $\vect{e}_i$ and $\vect{e}_j$ are
$\rho_0$-related coordinatewise for all distinct 
$i$ and $j$.
Since $\vect{h}=(h_1,\ldots,h_n) \in (\cl{R}_0^{(m)})^n$, 
the operations $h_1,\ldots,h_n$ preserve $\rho_0$.
Hence the $n$-tuples
$\vect{h}(\vect{e}_i)$ and
$\vect{h}(\vect{e}_j)$ are also
$\rho_0$-related coordinatewise for all 
distinct $i$ and $j$.

We will use the notation $\vect{0}$ for
tuples (of arbitrary length) whose 
coordinates are all $0$.
Since every operation $h_k$ 
($1\le k\le n$) preserves $\rho_0$,
and
$(0,0)\in\rho_0$ but $(1,1)\notin\rho_0$,
we get that $h_k(\vect{0})=0$.
Thus $\vect{h}(\vect{0})=\vect{0}$.
The following calculation shows that
$f_n\bigl(\vect{h}(\vect{e}_i)\bigr)\not=
f_n\bigl(\vect{0}\bigr)$
for each $i$:
\[
f_n\bigl(\vect{h}(\vect{e}_i)\bigr) = 
(f_n \circ \vect{h})(\vect{e}_i) = 
f_m(\vect{e}_i) \neq f_m(\vect{0}) = 
(f_n \circ \vect{h})(\vect{0}) = 
f_n\bigl(\vect{h}(\vect{0})\bigr) =
f_n(\vect{0}).
\]
Thus $\vect{h}(\vect{e}_i)\not=\vect{0}$
for each $i$.
Let $M$ denote the $0$--$1$ matrix
whose rows are the $n$-tuples 
$\vect{h}(\vect{e}_i)$  ($1 \leq i \le m$).
The fact that 
$\vect{h}(\vect{e}_i)$ and
$\vect{h}(\vect{e}_j)$ are 
$\rho_0$-related coordinatewise for all 
distinct $i$ and $j$ 
implies that each column of $M$ has at most
one occurrence of $1$.
The fact that each $\vect{h}(\vect{e}_i)$
is different from $\vect{0}$
implies that every row of $M$ has at least
one occurrence of $1$.
Since $M$ is $m\times n$, we get that $m\le n$.

This completes the proof of Theorem \ref{boolean-case}.
\end{proof}

As we mentioned in the introduction, 
the statement in Theorem~\ref{boolean-case}
fails for clones on a finite set $A$ with
more than two elements.
For these sets S{\l}upecki's clone is an example
of a clone that belongs 
to the filter $\FF_A$ (see \cite{leht}),
but is not a discriminator clone.
Therefore in this case the clone $\cl{D}$
generated by the discriminator function
is not the least element of $\FF_A$.
However, we can use Theorem~\ref{boolean-case}
to show that $\cl{D}$ is a minimal member of $\FF_A$.
This will also imply that for finite sets with
more than two elements the filter $\FF_A$ 
is not principal.

\begin{theorem}
\label{minimality}
Let $A$ be a finite set of cardinality $|A|>2$, 
and let $\cl{D}$ be the clone
generated by the discriminator function on $A$.
If $\cl{H}$ is a proper subclone of $\cl{D}$,
then $\fequiv[\cl{H}]$ has infinite index in $\cl{O}_A$.
\end{theorem}

\begin{proof}
It follows from Lemma~\ref{basic_props2}~(ii) that
the clones $\cl{C}$ for which $\fequiv[\cl{C}]$
has infinite index in $\cl{O}_A$ form an order ideal
in the lattice of all clones on $A$.
Therefore it suffices to prove the statement 
when $\cl{H}$ is a
maximal (proper) subclone of $\cl{D}$.

We may assume without loss of generality that
$B=\{0,1\}$ is a subset of $A$.
The operations in $\cl{D}$ preserve all subsets of
$A$, including $B$.
Therefore every operation $f \in\cl{D}$ 
can be restricted to $B$ to yield 
an operation $f|_B$ on $B$.
By a result of Marchenkov~\cite{Marchenkov} 
(see also \cite{Szendrei})
$\cl{D}$ has two maximal subclones:
\begin{itemize}
\item
the clone $\cl{E}$ consisting of all $f\in\cl{D}$
such that $f|_B$ is a linear operation on $B$, and
\item
the clone $\cl{K}$ consisting of all $f\in\cl{D}$
such that $f|_B$ is monotone with respect to
the order $0\le 1$ on $B$.
\end{itemize}
Thus $\cl{E}|_B$ is a subclone of the clone 
$\cl{L}$ of
linear operations on $B$, while
$\cl{K}|_B$ is 
a subclone of the clone $\cl{M}$ of
monotone operations on $B$.
Hence, by Theorem~\ref{boolean-case},
each one of the equivalence relations
$\fequiv[\cl{E}|_B]$ and $\fequiv[\cl{K}|_B]$
has infinite index in $\cl{O}_B$.
Therefore by Lemma~\ref{basic_props2}~(iii)
each one of $\fequiv[\cl{K}]$ and $\fequiv[\cl{E}]$
has infinite index in $\cl{O}_A$.
\end{proof}

\section{$\cl{C}$-equivalence for
discriminator clones $\cl{C}$ of Boolean functions}

Boolean functions are operations on the set $A=\{0,1\}$.
In this section we will explicitly describe the
$\cl{C}$-equivalence relation for Boolean functions
provided $\cl{C}$ is a discriminator clone.
We will also determine, for each such clone $\cl{C}$, 
the partial order $\preceq_\cl{C}$ induced on 
the set of $\cl{C}$-equivalence
classes by the quasiorder $\subf[\cl{C}]$.

To describe Boolean functions we will use the Boolean
algebra operations $\vee$, $\cdot$, 
and $\bar{\phantom{t}}$, 
as well as the Boolean ring operations $+$ and $\cdot$
on $A=\{0,1\}$. 
The unary constant operations will be denoted by
$0$ and $1$.
If $\vect{a}=(a_1,\ldots,a_n)$ is an $n$-tuple in $A^n$,
$\overline{\vect{a}}$ will denote the $n$-tuple
$(\bar a_1,\ldots,\bar a_n)$.
The tuples $(0,\ldots,0)$ and $(1,\ldots,1)$ will be
denoted by $\vect{0}$ and $\vect{1}$, respectively.

It is easy to see from Post's lattice (see \cite{Post}) 
or from the 
characterization of (locally closed) 
discriminator clones cited 
in Claim~\ref{bp} that 
there are six discriminator clones of Boolean
functions, namely
\begin{itemize}
\item
the clone $\cl{O}=\cl{O}_{A}$ 
of all Boolean functions;
\item
for each $i=0,1$, the clone $\cl{T}_i$ of all 
Boolean functions that fix $i$, that is,
\[
\cl{T}_0=\{f\in \cl{O}: f(\vect{0})=0\}
\qquad\text{and}\qquad
\cl{T}_1=\{f\in \cl{O}: f(\vect{1})=1\};
\]
\item
the clone $\cl{T}_\id=\cl{T}_0\cap\cl{T}_1$ of all
idempotent Boolean functions;
\item
the clone $\cl{S}$ of all self-dual Boolean functions, 
that is,
\[
\cl{S}=\bigl\{f\in \cl{O}: 
f(\overline{\vect{x}})=
\overline{f(\vect{x})}
\text{ for all $\vect{x}$}\bigr\};
\]
\item
the clone $\cl{D}=\cl{T}_\id\cap\cl{S}$ of all
idempotent self-dual Boolean functions,
\end{itemize}
and they are ordered by inclusion as shown in
Figure~\ref{fig-clones}.
\begin{figure} 
\begin{center}
\setlength{\unitlength}{0.5mm}
\begin{picture}(60,65)
\put(30,0){\circle*{4}}
\put(30,20){\circle*{4}}
\put(0,40){\circle*{4}}
\put(40,40){\circle*{4}}
\put(60,40){\circle*{4}}
\put(30,60){\circle*{4}}
\thicklines
\put(30,0){\line(-3,4){30}}
\put(30,0){\line(0,1){20}}
\put(30,20){\line(1,2){10}}
\put(30,20){\line(3,2){30}}
\put(30,60){\line(-3,-2){30}}
\put(30,60){\line(1,-2){10}}
\put(30,60){\line(3,-2){30}}
\put(34,59){$\cl{O}$}
\put(64,39){$\cl{T}_1$}
\put(29,39){$\cl{T}_0$}
\put(34,16){$\cl{T}_\id$}
\put(-10,39){$\cl{S}$}
\put(34,-1){$\cl{D}$}
\end{picture}
\end{center}
\caption{Discriminator clones of Boolean functions}
\label{fig-clones}
\end{figure}
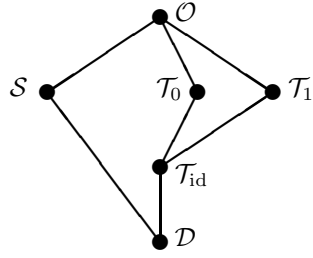

Our main tool in understanding $\cl{C}$-equivalence for
these clones $\cl{C}$ will be Theorem~\ref{le_C}.
To be able to apply the theorem
we will need to know the $\sim_\cl{C}$-blocks
in $A^n$ for each $n\ge1$, and the subalgebras $\al{S}^\cl{C}_\vect{a}$ of $(A;\cl{C})$ for all
$\vect{a}\in A^n$.
The descriptions of the six discriminator
clones above yield that
for each $\vect{a}\in A^n$ ($n\ge1$),
\begin{equation}
\label{orb}
\vect{a}/{\sim_\cl{C}}=
\begin{cases}
\{\vect{a},\overline{\vect{a}}\} &
\text{if $\cl{C}\subseteq\cl{S}$,}\\
\{\vect{a}\} &
\text{otherwise;}
\end{cases}
\end{equation}
and
\begin{equation}
\label{subal}
\al{S}^\cl{C}_\vect{a}=
\begin{cases}
\{0\} &
\text{if $\vect{a}=\vect{0}$ and
            $\cl{C}\subseteq\cl{T}_0$,}\\
\{1\} &
\text{if $\vect{a}=\vect{1}$ and
            $\cl{C}\subseteq\cl{T}_1$,}\\
\{0,1\} &
\text{otherwise.}
\end{cases}
\end{equation}
(\ref{orb}) implies that each $\sim_\cl{C}$-block
has the same size, which we will denote by $d_\cl{C}$;
namely, 
\[
d_\cl{C}=
\begin{cases}
2 &
\text{if $\cl{C}\subseteq\cl{S}$,}\\
1 &
\text{otherwise.}
\end{cases}
\]
Furthermore, 
\[
T^\cl{C}_n=
\begin{cases}
     \{\vect{c}=(c_1,\ldots,c_n)\in A^n: c_1=0\} &
\text{if $\cl{C}\subseteq\cl{S}$,}\\
A^n &
\text{otherwise}
\end{cases}
\]
is a transversal for the $\sim_\cl{C}$-blocks in $A^n$.

For arbitrary Boolean function $f$ let 
$\Im^{[2]}(f)$ denote the collection of all
sets of the form 
$\{f(\vect{a}),f(\overline{\vect{a}})\}$ as $\vect{a}$
runs over all elements of the domain of $f$, and let
$\Im^{[1]}(f)$ denote the collection of all singletons
$\{f(\vect{a})\}$ as $\vect{a}$
runs over all elements of the domain of $f$.
Thus $\Im^{[d_\cl{C}]}(f)$ consists of the ranges
of all functions $f|_P$ as $P$ runs over
all $\sim_\cl{C}$-blocks in the domain of $f$.

\begin{theorem}
\label{boolean_le}
Let $\cl{C}$ be a discriminator clone of Boolean
functions.
The following conditions on
$f\in\cl{O}^{(n)}$ and $g\in\cl{O}^{(m)}$
are equivalent:
\begin{enumerate}
\item[{\rm(a)}]
$f\subf[\cl{C}]g$;
\item[{\rm(b)}]
$f(\vect{0})=g(\vect{0})$ 
if $\cl{C}\subseteq\cl{T}_0$,  
$f(\vect{1})=g(\vect{1})$ 
if $\cl{C}\subseteq\cl{T}_1$, and
$\Im^{[d_\cl{C}]}(f)\subseteq \Im^{[d_\cl{C}]}(g)$.
\end{enumerate}
If $\cl{C}=\cl{T}_\id$, 
$\cl{T}_0$, $\cl{T}_1$,
or $\cl{O}$, then the inclusion 
$\Im^{[d_\cl{C}]}(f)\subseteq \Im^{[d_\cl{C}]}(g)$
in condition (b) can be replaced by 
$\Im(f)\subseteq \Im(g)$.
\end{theorem}
 
\begin{proof}
First we will prove the equivalence of conditions
(a) and (b).
By Theorem~\ref{le_C}, $f\subf[\cl{C}]g$ if and only if
for all $P=\vect{c}/{\sim_\cl C}$ with 
$\vect{c}\in T^\cl{C}_n$,
\begin{equation}
\label{set}
f|_P\in
\{g|_Q\circ\Phi^\cl{C}_{\vect{d},\vect{c}}:
Q=\vect{d}/{\sim_\cl{C}},\ 
\vect{d}\in A^m,\ 
\al{S}^\cl{C}_\vect{d}\le \al{S}^\cl{C}_\vect{c}
\}.
\end{equation}
The functions $\Phi^\cl{C}_{\vect{d},\vect{c}}
\colon P\to Q$
here are bijections, since they are surjective
by definition, and $|P|=|Q|=d_\cl{C}$. 
If $\al{S}^\cl{C}_\vect{c}=\{0\}$, then
$\vect{c}=\vect{0}$ and $\cl{C}\subseteq\cl{T}_0$
by (\ref{subal}).
Thus
$\al{S}^\cl{C}_\vect{d}\le \al{S}^\cl{C}_\vect{c}$
forces $\vect{d}=\vect{0}$.
Similarly, 
if $\al{S}^\cl{C}_\vect{c}=\{1\}$, then
$\vect{c}=\vect{1}$, $\cl{C}\subseteq\cl{T}_1$, and
$\vect{d}=\vect{1}$.
In all other cases
$\al{S}^\cl{C}_\vect{c}=\{0,1\}$, therefore
all $\vect{d}\in A^m$ satisfy 
$\al{S}^\cl{C}_\vect{d}\le \al{S}^\cl{C}_\vect{c}$. 
Since $|P|=d_\cl{C}=1$ or $2$, it follows that
in this case each bijection of
$P$ onto another $\sim_\cl{C}$-block $Q$ 
is of the form
$\Phi^\cl{C}_{\vect{d},\vect{c}}$ for an appropriate
$\vect{d}\in Q$.
Consequently, (\ref{set}) is equivalent to the following condition:
\begin{enumerate}
\item[(1)]
$f|_P =g|_Q\circ\phi$
for the unique bijection $\phi\colon P\to Q$
with $\phi(\vect{0})=\vect{0}$, 
if $P=\vect{0}/{\sim_\cl{C}}$ 
and $\cl{C}\subseteq\cl{T}_0$;
\item[(2)]
$f|_P =g|_Q\circ\phi$
for the unique bijection $\phi\colon P\to Q$
with $\phi(\vect{1})=\vect{1}$, 
if $P=\vect{1}/{\sim_\cl{C}}$
and $\cl{C}\subseteq\cl{T}_1$;
\item[(3)]
$f|_P \in\{g|_Q\circ\phi: 
\text{$\phi$ is a bijection $P\to Q$},\ 
Q=\vect{d}/{\sim_\cl{C}},\ 
\vect{d}\in A^m\}$ otherwise.
\end{enumerate}
(1) and (2) require that
\begin{enumerate}
\item[(i)]
$f(\vect{0})=g(\vect{0})$ holds if 
$\cl{C}\subseteq\cl{T}_0$
and $d_\cl{C}=1$
(that is, if 
$\cl{D}\not=\cl{C}\subseteq\cl{T}_0$),
\item[(ii)]
$f(\vect{1})=g(\vect{1})$ holds if 
$\cl{C}\subseteq\cl{T}_1$
and $d_\cl{C}=1$
(that is, if 
$\cl{D}\not=\cl{C}\subseteq\cl{T}_1$), and
\item[(iii)]
both of
$f(\vect{0})=g(\vect{0})$ and 
$f(\vect{1})=g(\vect{1})$ hold if
$\cl{C}\subseteq\cl{T}_i$ for $i=0$ or $1$ 
and 
$d_\cl{C}=2$
(that is, if 
$\cl{C}=\cl{D}\,(\subseteq \cl{T}_0\cap\cl{T}_1)$).
\end{enumerate}
In (3) the set
$\{g|_Q\circ\phi: 
\text{$\phi$ is a bijection $P\to Q$},\ 
Q=\vect{d}/{\sim_\cl{C}},\ 
\vect{d}\in A^m\}$
is equal to the set of functions
$P\to A$ whose range is in $\Im^{[d_\cl{C}]}(g)$. 
Therefore condition (3)
can be rephrased as follows:
\begin{enumerate}
\item[(iv)]
for all $P=\vect{c}/{\sim_\cl{C}}$ 
($\vect{c}\in T^\cl{C}_n$) not covered by 
(i)--(iii) 
$f|_P$ is a function $P\to A$ whose range is in 
$\Im^{[d_\cl{C}]}(g)$.
\end{enumerate}
It is easy to see now that (i)--(iv) hold for all $f|_P$
($P=\vect{c}/{\sim_\cl{C}}$, $\vect{c}\in T^\cl{C}_n$)
if and only if $f$ and $g$ satisfy condition (b).
This completes the proof of the equivalence of
conditions (a) and (b).

If $\cl{C}$ is one of the clones $\cl{T}_\id$,
$\cl{T}_0$, $\cl{T}_1$,
or $\cl{O}$, then $d_\cl{C}=1$.
Hence for each Boolean function $f\in\cl{O}$, 
$\Im^{[d_\cl{C}]}(f)$ is the set of singletons
$\{r\}$ with $r\in\Im(f)$.
Therefore for arbitrary $f,g\in\cl{O}$ we have
$\Im^{[d_\cl{C}]}(f)\subseteq \Im^{[d_\cl{C}]}(g)$
if and only if
$\Im(f)\subseteq \Im(g)$,
proving the last statement of the theorem.
\end{proof}

For each discriminator clone $\cl{C}$
of Boolean functions
Theorem~\ref{boolean_le}
allows us to describe 
the $\cl{C}$-equivalence classes
of Boolean functions
and also the partial order $\preceq_\cl{C}$ induced
by $\subf[\cl{C}]$ on the set $O/{\fequiv[\cl{C}]}$
of $\cl{C}$-equivalence classes.

We will use the following notation:
$N$ will denote the set of all nonconstant 
functions in $\cl{O}$, and $[i]$ ($i=0,1$) 
the set of all constant functions with value $i$.
For a nonempty
subset $R$ of $\bigl\{\{0\},\{1\},\{0,1\}\bigr\}$, 
$F_R$ will denote the set of all functions
$f\in\cl{O}$ such that $\Im^{[2]}(f)=R$.
Furthermore, for any ordered pair $(a,b)\in\{0,1\}^2$
and for any set $U$ of Boolean functions, $U^{ab}$
will denote the set of all functions $f\in U$
such that $f(\vect0)=a$ and $f(\vect1)=b$.
Notice that with this notation 
$[i]=F_{\{i\}}=F_{\{i\}}^{ii}$
($i=0,1$).

\def\oneclass#1{%
\begin{picture}(28,14)
\thicklines
\put(15,7.5){\oval(28,14)}
\put(1,5.5){\hbox to 14truemm{\hfil$#1$\hfil}}
\end{picture}
}

\def\oneclassbottom#1#2{%
\begin{picture}(28,14)
\thicklines
\put(15,7.5){\oval(28,14)}
\put(1,5.5){\hbox to 14truemm{\hfil$#1$\hfil}}
\put(1,-7){\hbox to 14truemm{\hfil$#2$\hfil}}
\end{picture}
}

\def\oneclasstop#1#2{%
\begin{picture}(28,14)
\thicklines
\put(15,7.5){\oval(28,14)}
\put(1,5.5){\hbox to 14truemm{\hfil$#1$\hfil}}
\put(1,18){\hbox to 14truemm{\hfil$#2$\hfil}}
\end{picture}
}

\def\onediamond#1#2#3#4#5#6#7#8{%
\begin{picture}(56,84)
\put(14,0){\oneclassbottom{#1}{#5}}
\put(28,28){\oneclass{#2}}
\put(0,42){\oneclass{#3}}
\put(14,70){\oneclasstop{#4}{#8}}
\put(32.5,14.5){\line(1,2){6.9}}
\put(17.5,56.5){\line(1,2){6.9}}
\put(25.7,14.4){\line(-1,3){9.3}}
\put(40.7,42.4){\line(-1,3){9.3}}
\put(42,21){$#6$}
\put(4,60){$#7$}
\end{picture}
}

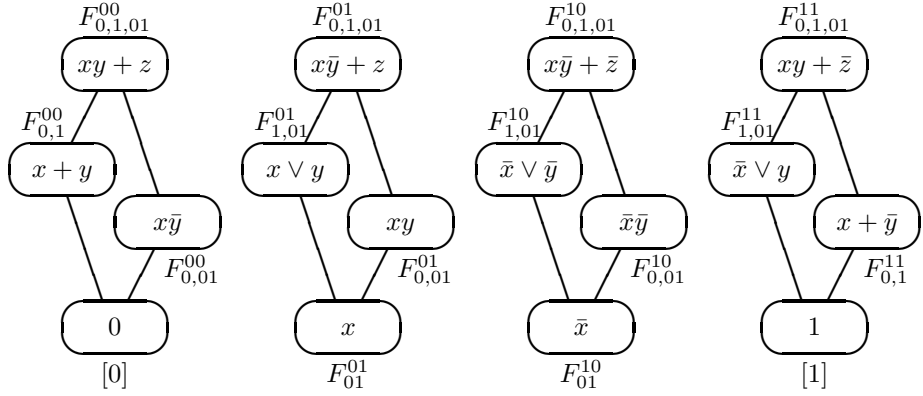
\begin{figure} 
\begin{center}
\setlength{\unitlength}{0.5mm}
\begin{picture}(250,105)
\thicklines
\put(0,7){\onediamond{0}{x\bar{y}}{x+y}{xy+z}%
{[0]}{F_{0,01}^{00}}{F_{0,1}^{00}}{F_{0,1,01}^{00}}}
\put(62,7){\onediamond{x}{xy}{x\vee y}{x\bar{y}+z}%
{F_{01}^{01}}{F_{0,01}^{01}}%
{F_{1,01}^{01}}{F_{0,1,01}^{01}}}
\put(124,7){\onediamond{\bar{x}}{\bar{x}\bar{y}}%
{\bar{x}\vee\bar{y}}{x\bar{y}+\bar{z}}%
{F_{01}^{10}}{F_{0,01}^{10}}%
{F_{1,01}^{10}}{F_{0,1,01}^{10}}}
\put(186,7){\onediamond{1}{x+\bar{y}}%
{\bar{x}\vee y}{xy+\bar{z}}%
{[1]}{F_{0,1}^{11}}{F_{1,01}^{11}}{F_{0,1,01}^{11}}}
\end{picture}
\end{center}
\caption{The poset 
$(\cl{O}/{\fequiv[\cl{D}]};\preceq_\cl{D})$}
\label{fig-D}
\end{figure}

It follows from Theorem~\ref{boolean_le} 
that $f\fequiv[\cl{D}]g$ if and only if 
$f(\vect{0})=g(\vect{0})$,   
$f(\vect{1})=g(\vect{1})$, and 
$\Im^{[2]}(f)=\Im^{[2]}(g)$.
Therefore the $\cl{D}$-equivalence
classes are the nonempty sets of the form 
$F_R^{ab}$ where 
$\emptyset\not=R\subseteq\bigl\{\{0\},\{1\},\{0,1\}\bigr\}$
and $(a,b)\in\{0,1\}^2$. 
If $F_R^{ab}\not=\emptyset$, then $\{a,b\}\in R$,
because 
$f\in F_R^{ab}$ implies that $R=\Im^{[2]}(f)$ and
$\{a,b\}=\{f(\vect{0}),f(\vect{1})\}\in\Im^{[2]}(f)$.
Thus the $\cl{D}$-equivalence classes are
the nonempty sets among the $16$ sets $F_R^{ab}$ with
$\{a,b\}\in R\subseteq\bigl\{\{0\},\{1\},\{0,1\}\bigr\}$.
Figure~\ref{fig-D} shows these $16$ 
sets along with representatives for each of 
them, proving that none of them are empty.
Hence there are 16 $\cl{D}$-equivalence classes,
and according to Theorem~\ref{boolean_le}, 
the ordering $\preceq_\cl{D}$ between them is
as depicted in Figure~\ref{fig-D}.
For notational simplicity, in Figure~\ref{fig-D}
we omit braces when we list the elements of $R$ in 
$F_R^{ab}$.
For example, we write $F_{0,01}^{10}$ instead of $F_{\{\{0\},\{0,1\}\}}^{10}$.

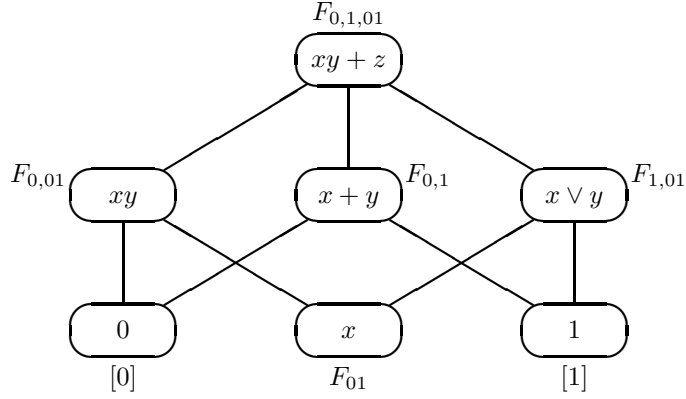
\begin{figure} 
\begin{center}
\setlength{\unitlength}{0.5mm}
\begin{picture}(148,105)
\thicklines
\put(0,7){\oneclassbottom{0}{[0]}}
\put(60,7){\oneclassbottom{x}{F_{01}}}
\put(120,7){\oneclassbottom{1}{[1]}}
\put(0,43){\oneclass{xy}}
\put(60,43){\oneclass{x+y}}
\put(120,43){\oneclass{x\vee y}}
\put(60,79){\oneclasstop{xy+z}{F_{0,1,01}}}
\put(75,57.5){\line(0,1){22}}
\put(15,21.5){\line(0,1){22}}
\put(135,21.5){\line(0,1){22}}
\put(26,21.1){\line(5,3){38}}
\put(26,57.1){\line(5,3){38}}
\put(86,21.1){\line(5,3){38}}
\put(64,21.1){\line(-5,3){38}}
\put(124,21.1){\line(-5,3){38}}
\put(124,57.1){\line(-5,3){38}}
\put(-15,54){$F_{0,01}$}
\put(90,54){$F_{0,1}$}
\put(150,54){$F_{1,01}$}
\end{picture}
\end{center}
\caption{The poset 
$(\cl{O}/{\fequiv[\cl{S}]};\preceq_\cl{S})$}
\label{fig-S}
\end{figure}

For the clone $\cl{S}$
Theorem~\ref{boolean_le} yields 
that $f\fequiv[\cl{S}]g$ if and only if 
$\Im^{[2]}(f)=\Im^{[2]}(g)$.
Thus the $\cl{S}$-equivalence classes are
the nonempty sets among the $7$ sets $F_R$ with
$\emptyset\not=
R\subseteq\bigl\{\{0\},\{1\},\{0,1\}\bigr\}$.
Figure~\ref{fig-S} shows these  
sets together with representatives for each of 
them, hence none of them are empty.
Thus there are $7$ $\cl{S}$-equivalence classes,
and according to Theorem~\ref{boolean_le}, 
the ordering $\preceq_\cl{S}$ between them is
as indicated in Figure~\ref{fig-S}.

\begin{figure} 
\begin{center}
\setlength{\unitlength}{0.5mm}
\begin{picture}(154,65)
\thicklines
\put(0,7){\oneclassbottom{0}{[0]}}
\put(126,7){\oneclassbottom{1}{[1]}}
\put(0,42){\oneclasstop{x+y}{N^{00}}}
\put(42,42){\oneclasstop{x}{N^{01}}}
\put(84,42){\oneclasstop{\bar x}{N^{10}}}
\put(126,42){\oneclasstop{x+\bar y}{N^{11}}}
\put(15,21.5){\line(0,1){21}}
\put(141,21.5){\line(0,1){21}}
\end{picture}
\end{center}
\caption{The poset 
$(\cl{O}/{\fequiv[\cl{T}_\id]};\preceq_{\cl{T}_\id})$}
\label{fig-Tid}
\end{figure}

Proceeding similarly, for the clone $\cl{T}_\id$ 
we get from Theorem~\ref{boolean_le} 
that $f\fequiv[\cl{T}_\id]g$ if and only if 
$f(\vect{0})=g(\vect{0})$,   
$f(\vect{1})=g(\vect{1})$, and 
$\Im(f)=\Im(g)$.
Since the range of each nonconstant Boolean
function is $\{0,1\}$, we conclude that the 
$\cl{T}_\id$-equivalence classes are
$[0]$, $[1]$, and $N^{ab}$ with $(a,b)\in\{0,1\}^2$.
Figure~\ref{fig-Tid} shows representatives of
these classes and the ordering $\preceq_{\cl{T}_\id}$
among them according to Theorem~\ref{boolean_le}.

\begin{figure} 
\begin{center}
\setlength{\unitlength}{0.5mm}
\begin{picture}(70,65)
\thicklines
\put(0,7){\oneclassbottom{0}{[0]}}
\put(42,7){\oneclassbottom{1}{[1]}}
\put(0,42){\oneclasstop{x}{N^{0*}}}
\put(42,42){\oneclasstop{\bar x}{N^{1*}}}
\put(15,21.5){\line(0,1){21}}
\put(57,21.5){\line(0,1){21}}
\end{picture}
\end{center}
\caption{The poset 
$(\cl{O}/{\fequiv[\cl{T}_0]};\preceq_{\cl{T}_0})$}
\label{fig-T0}
\end{figure}
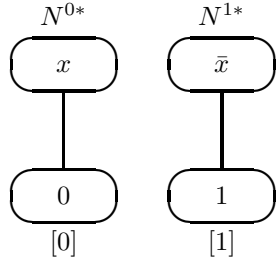

Analogously, Theorem~\ref{boolean_le} yields that
the $\cl{T}_0$-equivalence classes 
are $[i]$ and  $N^{i*}=N^{i0}\cup N^{i1}$ ($i=0,1$)
with representatives and
ordering as shown in Figure~\ref{fig-T0}.
The results for $\cl{T}_1$ are similar.

\begin{figure} 
\begin{center}
\setlength{\unitlength}{0.5mm}
\begin{picture}(70,65)
\thicklines
\put(0,7){\oneclassbottom{0}{[0]}}
\put(42,7){\oneclassbottom{1}{[1]}}
\put(21,42){\oneclasstop{x}{N}}
\put(19.2,21.5){\line(3,5){12.5}}
\put(52.8,21.5){\line(-3,5){12.5}}
\end{picture}
\end{center}
\caption{The poset 
$(\cl{O}/{\fequiv[\cl{O}]};\preceq_{\cl{O}})$}
\label{fig-O}
\end{figure}
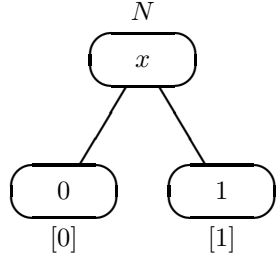

Finally, we obtain either 
from Theorem~\ref{boolean_le} or from
the special case $|A|=2$ of Corollary~\ref{cor-O} that
the $\cl{O}$-equivalence classes 
are $[i]$ ($i=0,1$) and $N$
with representatives and
ordering as shown in Figure~\ref{fig-O}.

To conclude our discussion of the posets 
$(\cl{O}/{\fequiv[\cl{C}]};\preceq_{\cl{C}})$,
recall from Corollary~\ref{cor-nu} that if
$\cl{C}\subseteq\cl{C}'$, then we have a
surjective, order preserving mapping 
$\nu_{\cl{C}',\cl{C}}$
from the poset 
$(\cl{O}_A/{\fequiv[\cl{C}]};\preceq_\cl{C})$
onto 
$(\cl{O}_A/{\fequiv[\cl{C}']};\preceq_{\cl{C}'})$,
which assigns to each $\cl{C}$-equivalence class
the $\cl{C}'$-equivalence class containing it.
By (\ref{comp-nu}) it suffices to look at the mappings
$\nu_{\cl{C}',\cl{C}}$ for covering pairs 
$\cl{C}\subset\cl{C}'$.

For each covering pair $\cl{C}\subset\cl{C}'$
of discriminator clones (see Figure~\ref{fig-clones}),
one can read off of Figures~\ref{fig-D}--\ref{fig-O}
what the corresponding natural mapping 
$\nu_{\cl{C}',\cl{C}}$ is.
For example, the mapping
$\nu_{\cl{S},\cl{D}}\colon 
(\cl{O}_A/{\fequiv[\cl{D}]};\preceq_\cl{D})\to
(\cl{O}_A/{\fequiv[\cl{S}]};\preceq_{\cl{S}})$
preserves the heights of elements, and
\begin{itemize}
\item
for elements of height $0$, it sends
$[i]$ to $[i]$ ($i=0,1$), and the other two elements 
$F_{01}^{01}, F_{01}^{10}$ 
in Figure~\ref{fig-D} to the middle
element $F_{01}$ in Figure~\ref{fig-S};
\item
for elements of height $1$, it sends
the leftmost and rightmost elements 
$F_{0,1}^{00}, F_{0,1}^{11}$
in Figure~\ref{fig-D} 
to the middle element $F_{0,1}$ in Figure~\ref{fig-S},
and among the remaining six elements in Figure~\ref{fig-D},
it sends the three that appear lower
to the leftmost element $F_{0,01}$ in Figure~\ref{fig-S},
and the three that appear
higher to the rightmost element $F_{1,01}$ in Figure~\ref{fig-S};
\item
finally, it sends
all four elements of height $2$ in Figure~\ref{fig-D}
to the largest element in Figure~\ref{fig-S}.
\end{itemize}
The natural mapping 
$\nu_{\cl{T}_\id,\cl{D}}\colon 
(\cl{O}_A/{\fequiv[\cl{D}]};\preceq_\cl{D})\to
(\cl{O}_A/{\fequiv[\cl{T}_\id]};\preceq_{\cl{T}_\id})$
preserves the four connected components, and
\begin{itemize}
\item
in the first and last connected components it sends
$[i]$ to $[i]$ ($i=0,1$), and the remaining three
elements in Figure~\ref{fig-D} to 
the height $1$ element $N^{ii}$ ($i=0,1$) of the corresponding component
in Figure~\ref{fig-Tid};
\item
in the second and third connected components
it sends all four elements in Figure~\ref{fig-D} to 
the unique element of the corresponding component in Figure~\ref{fig-Tid}.
\end{itemize}

\section*{Acknowledgements}
This work was initiated while the first author
was visiting the University of Waterloo. 
He is indebted to professors Ian Goulden and 
Bruce Richmond for helpful discussions.

\end{document}